\documentclass[a4paper,10pt, reqno]{amsart}

\usepackage[backref]{hyperref}
\usepackage{a4wide}
\usepackage{amsthm}
\usepackage{amssymb}
\usepackage{palatino}
\usepackage{tcolorbox}
\usepackage{amssymb}
\newtheorem{thm}{Theorem}

\newtheorem{prop}{Proposition}

\newtheorem{lemma}{Lemma}

\newtheorem{remark}{Remark}
\makeatletter
\newcommand*{\rom}[1]{\expandafter\@slowromancap\romannumeral #1@}
\def\blfootnote{\gdef\@thefnmark{}\@footnotetext}
\makeatother
\def\house#1{\setbox1=\hbox{$\,#1\,$}
\dimen1=\ht1 \advance\dimen1 by 2pt \dimen2=\dp1 \advance\dimen2 by 2pt
\setbox1=\hbox{\vrule height\dimen1 depth\dimen2\box1\vrule}%
\setbox1=\vbox{\hrule\box1}%
\advance\dimen1 by .4pt \ht1=\dimen1
\advance\dimen2 by .4pt \dp1=\dimen2 \box1\relax}

\begin{document}
\title{Primes in Beatty sequence}
\author{C. G. KARTHICK BABU}
\date{}
\address[]{Institute of Mathematical Science, HBNI
\\ C.I.T Campus, Taramani \\ 
   Chennai 600113.}
\email[C. G. Karthick Babu]{cgkbabu@imsc.res.in}
\subjclass[2010]{11B83,11N13,11L07}
\keywords{Beatty sequence, Prime number, Estimates on exponential sums}
 
\begin{abstract}
For a polynomial $g(x)$ of  $\deg k \geq 2$ with integer coefficient, we prove an upper bound for the least prime $p$ such that $g(p)$ is in an irrational non-homogeneous Beatty sequence $\lbrace \lfloor \alpha n+\beta\rfloor : n=1,2,3, \dots \rbrace$, where $\alpha, \beta \in \mathbb{R}$ with $\alpha >1$ and we prove an asymptotic formula for the number of primes $p$ such that $g(p)=\lfloor \alpha n+\beta \rfloor.$ Next we obtain an asymptotic formula for number of primes $p$ of the form $p=\lfloor \alpha n+\beta \rfloor$ which also satisfies $p \equiv f \pmod d$ where $\alpha,\beta$ are real numbers, $\alpha$ is irrational and f, d are integers with $1\leq f < d$ and $(f,d)=1$.
\end{abstract}

\maketitle
\section{Introduction}\label{intro}
Given a real number $\alpha >0$ and a non-negative real $\beta$, the Beatty sequence associated with $\alpha, \beta$ is defined by
$$\mathcal{B}(\alpha,\beta)=\lbrace\lfloor n\alpha+\beta \rfloor : n \in \mathbb{N}\rbrace,$$
where $\lfloor x\rfloor$ denotes the largest integer less than or equal to $x$. If $\alpha$ is rational, then $\mathcal{B}(\alpha,\beta)$ is union of residue classes, hence \textit{we always assume that $\alpha$ is irrational}. An irrational number $\gamma$ is said to be of finite type $t \geq 1$ if
$$t= \sup \lbrace \rho \in \mathbb{R} : \liminf_{n \rightarrow \infty} n^{\rho}\Vert n\gamma \Vert=0 \rbrace,$$
where $\Vert x \Vert$ is the distance of $x$ from nearest integer. In 2016, J\"{o}rn Steuding and Marc Technau \cite{ST16} proved that, for every $\varepsilon > 0$ there exists a computable positive integer $l$ such that for every irrational $\alpha >1$ the least prime $p$ in the Beatty sequence $\mathcal{B}(\alpha,\beta)$ is satisfies the inequality $$p \leq L^{35-16\varepsilon} \alpha^{2(1-\varepsilon)}B p_{m+l}^{1+\varepsilon},$$ where $B= \max\lbrace 1, \beta \rbrace$, $L= \log (2\alpha B),$ $p_{n}$ denotes the numerator of the $n^{th}$ convergent to the regular continued fraction expansion of 
$\alpha=[a_{0},a_{1},\dots]$ and m is the unique integer such that
$$p_{m}\leq L^{16}\alpha^{2} <p_{m+1}.$$

\vspace{2mm}
\noindent
The first result in this paper is the following
\begin{thm} \label{thm1}
Let $g(x)= a_{0}+a_{1}x+\dots +a_{k}x^{k},$ where $a_{0}, \dots , a_{k} \in \mathbb{Z}$ with $a_{k} \geq 1$ and $k \geq 2$. Put $\gamma= 4^{1-k}.$ Then for any positive integer $N \geq 3,$ positive real number $\alpha$ with $\bigg|\frac{a_{k}}{\alpha}- \frac{a}{q}\bigg| \leq \frac{1}{q^{2}},$ $(a,q)=1$ and any $\varepsilon >0$ we have
\begin{align*}
\sum_{\substack{p \leq N \\ g(p) \in \mathcal{B}(\alpha, \beta)}}\log p = \frac{1}{\alpha}\sum_{p \leq N}\log p +O\bigg(N^{\varepsilon}(N q^{-\gamma}+ N^{1-\gamma/2}+q^{\frac{\gamma}{1-\gamma}}N^{\frac{1-(k+1)\gamma}{1-\gamma}}+q^{\gamma}N^{1-k\gamma})\bigg).
\end{align*}

In particular, if $a_{k}/\alpha$ is an irrational number of finite type $t>0$ then we have
\begin{align*}
\sum_{\substack{p \leq N \\ g(p) \in \mathcal{B}(\alpha, \beta)}}\log p = \frac{1}{\alpha}\sum_{p \leq N}\log p +O(N^{1-\frac{k\gamma}{t+1}+\varepsilon}+ N^{1-\frac{\gamma}{2}+\varepsilon})
\end{align*}
\end{thm}
\vspace{2mm}
\noindent

\begin{thm} \label{thm2}
Let $g(x)= a_{0}+a_{1}x+\dots +a_{k}x^{k},$ where $a_{0}, \dots , a_{k} \in \mathbb{Z}$ with $a_{k} \geq 1$ and $k \geq 2$. Put $\gamma= 4^{1-k}.$ For every $\varepsilon > 0,$ there exits a computable positive integer l such that for every irrational $\alpha > 1$ the least prime number $p$ such that $g(p)$ is contained in the Beatty sequence $\mathcal{B}(\alpha,\beta)$ satisfies
the inequality,
\begin{equation*}
p \leq \alpha^{\frac{2k-1}{k\gamma}-\frac{(\gamma+1)\varepsilon}{\gamma}} B^{\frac{k-1}{k}-\frac{\varepsilon}{\gamma}} p_{m+l}^{\frac{1}{k}+\frac{\varepsilon}{\gamma}},
\end{equation*} 
where $B= \max\lbrace 1, \beta \rbrace$, 
$p_{n}$ denotes the numerator of the $n^{th}$ convergent to the regular continued fraction expansion of 
$\frac{\alpha}{a_{k}}$ and m is the unique integer such that
\begin{equation}\label{cngtineq}
p_{m}\leq \alpha^{\frac{1+\gamma}{\gamma}}B <p_{m+1}.
\end{equation}   
\end{thm}

\vspace{3mm}
For irrational $\alpha$ of finite type $\tau= \tau(\alpha)$, Banks and Yeager proved in (\cite{BY11}, Theorem 2) that for any fixed $\varepsilon > 0,$ for all integers $1 \leq c < d < N^{\frac{1}{4\tau+2}}$ with gcd (c, d)=1, we have 
\begin{equation*}
\sum\limits_{\substack{n \leq x, n \in \mathcal{B}_{\alpha,\beta} \\ n \equiv c \pmod d}} \Lambda(n) = \frac{1}{\alpha} \sum_{\substack{n \leq x \\ n \equiv c \pmod d}}\Lambda(n)+ O(x^{1-\frac{1}{4\tau+2}+\varepsilon}).
\end{equation*}

\vspace{2mm}
\noindent
The following theorem improves the error term.
\begin{thm} \label{thm3}
For any positive integers $N \geq 3$, $1 \leq f <d \leq \min\lbrace q^{1/2}, N^{1/6}\rbrace$ such that $(f,d)=1$ and positive real number $\alpha$ with $\bigg|\frac{1}{\alpha}- \frac{a}{q}\bigg| \leq \frac{1}{q^{2}}$ with $(a,q)=1,$ for any $\varepsilon >0$, we have
\begin{align*}
\sum_{\substack{p \leq N \\ p \in \mathcal{B}(\alpha, \beta)\\ p \equiv f (d)}}\log p =\frac{1}{\alpha}\sum_{\substack{p \leq N \\ p \equiv f(d)}} \log p
+O\bigg(N^{\varepsilon}\bigg(\frac{N}{q^{1/2}}+N^{1/2}q^{1/2}+N^{3/4}d^{1/2}+\frac{N^{4/5}}{d^{1/5}}\bigg)\bigg).
\end{align*}
Furthermore if $\frac{1}{\alpha}$ is an irrational number of finite type $t >0$ then for all integers $1 \leq f < d \leq \min \lbrace N^{\frac{1}{2(t+1)}}, N^{1/6} \rbrace$ with $(f,d)=1$ and for any $0< \varepsilon <\frac{1}{4(t+1)},$ we have
\begin{align*}
\sum_{\substack{p \leq N \\ p \in \mathcal{B}(\alpha, \beta)\\ p \equiv f (d)}}\log p =\frac{1}{\alpha}\sum_{\substack{p \leq N \\ p \equiv f(d)}} \log p
+O(N^{1-\frac{1}{2(1+t)}+\varepsilon}+ N^{\frac{3}{4}+\varepsilon}d^{1/2}+N^{\frac{4}{5}+\varepsilon}d^{-1/5}).
\end{align*}
\end{thm} 

\vspace{2mm}
\noindent
The proof of Theorem \ref{thm3} depends on estimation of exponential sum of the type
\begin{equation}\label{expsum}
S(\vartheta)=\sum_{|l| \leq L} \bigg|\sum_{\substack{n\leq N \\ n \equiv f(d)}} \Lambda(n)e(ln \vartheta)\bigg|,
\end{equation}
where $\vartheta$ is irrational, $L, N \geq 1$ and $f<d, (f,d)=1.$ 

\noindent
We obtain an upper bound for $S(\vartheta)$ in Proposition \ref{teclemma2} which is of independent interest. 

\begin{remark}\label{rmk 1}
Let $d,f$ be natural numbers such that $ 1\leq f<d \leq 500$ and $(f,d)=1$. For every $\varepsilon > 0$ there exists a computable positive integer $l$ such that for every irrational $\alpha >1;$ the least prime number $p \in  \mathcal{B}(\alpha,\beta)$ such that $p \equiv f \pmod d $ satisfies the inequality, 
\begin{equation*}
p \leq \alpha^{3-7\varepsilon} B^{\frac{1}{2}(1-3\varepsilon)} d^{3-10\varepsilon} p_{m+l}^{1+3\varepsilon}
\end{equation*}
where $B=$max$\lbrace 1,\beta \rbrace$ and $p_{n}$ denotes the numerator of the $n^{th}$ convergent to the regular continued fraction expansion of $\alpha$ and m is the unique integer such that,
\begin{equation*}
p_{m} \leq \alpha^{7/3} B^{1/2} d^{10/3}  < p_{m+1}.
\end{equation*}
\end{remark}

\vspace{2mm} 
\noindent 
This fact can be proved in a similar way as Theorem \ref{thm2} using Corollary 1.6 of \cite{BMK18}.

\vspace{3mm}
\textbf{Acknowledgement}. I would like to express my sincere thanks to my thesis supervisor Anirban Mukhopadhyay for his valuable and constructive suggestions during the planning and development of this paper. I would also like to thank Marc Technau for suggesting important changes in an earlier version of this manuscripts.
\section{Notation}\label{notation}
Throughout this paper, the implied constants in the symbols $O$ and $\ll$ may depend on $\alpha$ and $\varepsilon$ otherwise are absolute. We recall that the notation $f = O(g)$ and $f \ll g$ are equivalent to the assertion that the inequality $|f| \leq c g $ holds for some constant $c > 0$. The notation $f \approx g$ means that $f \ll g$ and $f \gg g$. It is important to note that our bounds are uniform with respect to all of the involved parameters other than $\alpha, \varepsilon$ and degree of the polynomial $k$; in particular, our bounds are uniform with respect to $\beta$.

\vspace{2mm}
\noindent
The letters $a, d, f, q $ always denote non-negative integers and $m, n, l, u, v$ and $t$ denotes integers. We use $\lfloor x \rfloor $ and $\lbrace x \rbrace$ to denote the greatest integer less than or equal to $x$ and the fractional part of $x$ respectively. Finally, recall that the discrepancy $D(M)$ of a sequence of (not necessarily distinct) real numbers $a_{1},a_{2} \dots , a_{M} \in [0, 1)$ is defined by
$$D(M)= \sup_{I \subset [0,1)}\bigg|\frac{V(I,M)}{M}-|I|\bigg|,$$
where the supremum is taken over all sub-intervals $I$ of $[0,1)$, $V(I,M)$ is the number of positive integers $m \leq M$ such that $a_{m} \in I$ and $|I|$ is the length of $I$.
\section{Preliminaries}\label{prelim}
\subsection{Case of polynomial values of prime}
Note that an integer $m \in \mathcal{B}(\alpha,\beta)$ if and only if\\ $\frac{m}{\alpha} \in \bigg(\frac{\beta-1}{\alpha},\frac{\beta}{\alpha}\bigg] \pmod 1$ and $m > \alpha+\beta-1.$ This is equivalent to
$$\bigg\Vert \frac{m}{\alpha}+\frac{1-2\beta}{2\alpha}\bigg\Vert < \frac{1}{2\alpha}.$$ 
Hence 
\begin{equation*}
\# \lbrace m \leq N : m \in \mathcal{B}(\alpha,\beta)\rbrace =\sum_{ m \leq N }\chi_{\frac{1}{2\alpha}}\bigg(\frac{m}{\alpha}+\frac{1-2\beta}{2\alpha}\bigg)
\end{equation*}
where $\chi_{\delta}$ for $\delta >0$ is defined by,
\begin{equation*}
\chi_{\delta}(\theta)=
\begin{cases}
    1 & \text{if } \Vert\theta\Vert < \delta,\\    
     0 & \text{otherwise}
\end{cases}
\end{equation*}
for $\theta \in \mathbb{R}$. Let $g(x)= a_{0}+a_{1}x+\dots +a_{k}x^{k},$ where $a_{0}, \dots , a_{k} \in \mathbb{Z}$, $a_{k} \geq 1$. Therefore
\begin{equation*}
\# \lbrace p \leq N : g(p) \in \mathcal{B}(\alpha,\beta)\rbrace =\sum_{ p \leq N }\chi_{\frac{1}{2\alpha}}\bigg(\frac{g(p)}{\alpha}+\frac{1-2\beta}{2\alpha}\bigg).
\end{equation*}
\begin{lemma}(\cite{HarmanBook}, Lemma 2.1)\label{lem1}.
For any $L \in \mathbb{N}$ there are coefficients $C^{\pm}_{l}$ such that 
\begin{equation*}
2\delta-\frac{1}{L+1}+\sum_{1 \leq l \leq L}C_{l}^{-}e(l\theta)\leq \chi_{\delta}(\theta) \leq 2\delta+\frac{1}{L+1}+\sum_{1 \leq l \leq L}C_{l}^{+}e(l\theta),
\end{equation*}
with $|C_{l}^{\pm}| \leq \min \bigg(2\delta+\frac{1}{L+1}, \frac{3}{2l}\bigg).$
\end{lemma} 

\noindent
Using Lemma \ref{lem1} we get
\begin{align}\label{fourier1}             
\sum_{ n\leq N }\Lambda(n)\chi_{\frac{1}{2\alpha}}\bigg(\frac{g(n)}{\alpha}+\frac{1-2\beta}{2\alpha}\bigg)=&\frac{1}{\alpha}\sum_{ n\leq N }\Lambda(n)+O\bigg(\frac{N}{L+1}\bigg)\nonumber\\ 
&+O\bigg(\sum_{1 \leq |l| \leq L}|C_{l}|\bigg\vert\sum_{n \leq N}\Lambda(n)e\bigg(\frac{lg(n)}{\alpha}\bigg)\bigg\vert\bigg),
\end{align}
where
$|C_{l}| \leq \min\bigg(\frac{1}{\alpha}+\frac{1}{L+1},\frac{3}{2l}\bigg).$
To estimate the exponential sum we use the following Proposition
\begin{prop}\label{teclemma1}(Equation (22), \cite{HARMAN})
Suppose $\varepsilon > 0$ is given. Let $f(x)$ be a real valued polynomial in $x$ of degree $k \geq 2.$ Put $\gamma= 4^{1-k}.$ Suppose $\alpha$ is the leading coefficient of $f$ and there are integers $a,q$ with $(a,q)=1$ such that 
$$|q\alpha-a| <q^{-1}.$$
Then we have
\begin{equation*}
\sum_{l \leq L}\bigg|\sum_{n \leq N} \Lambda(n) e(lf(n))\bigg| \ll (NL)^{1+\varepsilon} (q^{-1}+N^{-1/2}+qN^{-k}L^{-1})^{\gamma}.
\end{equation*}
\end{prop}

\noindent
\subsection{Case of primes in arithmetic progression}
Now we are interested in prime numbers p of the form $p \equiv f$ mod d, which is in $\mathcal{B}(\alpha,\beta)$, where $(f,d)=1$ and $f<d$. As we discussed above, in order to find a prime number $p \in \mathcal{B}(\alpha,\beta)$ and we need to show that
$$\bigg\Vert \frac{p}{\alpha}+\frac{1-2\beta}{2\alpha}\bigg\Vert < \frac{1}{2\alpha}.$$ 

\vspace{2mm}
\noindent
By Lemma \ref{lem1}, we have
\begin{align}\label{fourier2}
\sum_{\substack{n\leq N \\ n \equiv f(d)}} \Lambda(n)\chi_{\frac{1}{2\alpha}}\bigg(\frac{n}{\alpha}+\frac{1-2\beta}{2\alpha}\bigg)=&\frac{1}{\alpha}\sum_{\substack{n\leq N \\ n \equiv f(d)}} \Lambda(n)+O\bigg(\frac{N}{L\varphi(d)}\bigg)\nonumber\\
&+O\bigg(\sum_{1 \leq |l| \leq L}|C_{l}| \sum_{\substack{n\leq N \\ n \equiv f(d)}} \Lambda(n)e(ln/\alpha)\bigg),
\end{align}
where $|C_{l}| \leq \min\bigg(\frac{1}{\alpha}+\frac{1}{L+1},\frac{3}{2l}\bigg).$ Now we want to estimate the exponential sum of the form \eqref{expsum}. To estimate the exponential sum we use the following Proposition

\begin{prop}\label{teclemma2}
Let $S(\vartheta)$ is defined by \ref{expsum} with $\bigg|\vartheta-\frac{a}{q}\bigg| \leq q^{-2},$ where $a$ and $q$ are positive integers satisfying (a,q)=1. Then   for any real number $\varepsilon>0;$ we have
\begin{equation}\label{prop2eq}
S(\vartheta) \ll_{\varepsilon} (NL)^{\varepsilon} \bigg(\frac{NL}{q^{1/2}}+L^{1/2}N^{1/2}q^{1/2}+LN^{3/4}d^{1/2}+\frac{LN^{4/5}}{d^{1/5}}\bigg).
\end{equation}
\end{prop}

\noindent
We will give the proof of Proposition \ref{teclemma2} in Section \ref{section6}.

\section{Proof of Theorem 1 and Theorem 2}
In the previous section we stated essential results to prove Theorem \ref{thm1} and Theorem \ref{thm2}. In this section we will give proof of these theorems.
\vskip 1mm \noindent
{\it Proof of Theorem \ref{thm1}:} 
It follows from Proposition \ref{teclemma1} and partial summation formula
\begin{align} \label{prop1eq}
\sum_{1 \leq |l| \leq L}C_{l}\sum_{n \leq N}\Lambda(n)e\bigg(l\bigg(\frac{g(n)}{\alpha}+\frac{1-2\beta}{2\alpha}\bigg)\bigg) \ll_{\varepsilon}& N^{1+\varepsilon}  L^{\varepsilon}(q^{\gamma}+N^{-\gamma/2}+q^{\gamma}N^{-k\gamma}L^{-\gamma})\nonumber\\
&+N^{1+\varepsilon}q^{\gamma}N^{-k\gamma}.
\end{align}

\vspace{2mm}
\noindent
By \eqref{fourier1} and \eqref{prop1eq} we have
\begin{align*}
&\sum_{ n\leq N }\Lambda(n)\chi_{\frac{1}{2\alpha}}\bigg(\frac{g(n)}{\alpha}+\frac{1-2\beta}{2\alpha}\bigg) =\bigg(\frac{1}{\alpha}\bigg)\sum_{n \leq N}\Lambda(n)+O\bigg(\frac{N}{L+1}\bigg)\\
&+O\bigg(N^{1+\varepsilon}  L^{\varepsilon}(q^{-\gamma}+N^{-\gamma/2}+q^{\gamma}N^{-k\gamma}L^{-\gamma})
+N^{1+\varepsilon}q^{\gamma}N^{-k\gamma}\bigg).
\end{align*}
Choosing $L=q^{\frac{-\gamma}{1-\gamma}}N^{\frac{k\gamma}{1-\gamma}}$ we have
\begin{align*}
&\sum_{ n\leq N }\Lambda(n)\chi_{\frac{1}{2\alpha}}\bigg(\frac{g(n)}{\alpha}+\frac{1-2\beta}{2\alpha}\bigg) =\bigg(\frac{1}{\alpha}\bigg)\sum_{n \leq N}\Lambda(n)\\
&+O\bigg(N^{\varepsilon}(Nq^{-\gamma}+N^{1-\gamma/2}+q^{\gamma}N^{1-k\gamma}
+q^{\frac{\gamma}{1-\gamma}}N^{\frac{1-(k+1)\gamma}{1-\gamma}})\bigg).
\end{align*}
This leads to
\begin{equation}\label{thm1eq1}
\sum_{\substack{p^{\nu} \leq N \\ g(p^{\nu}) \in \mathcal{B}(\alpha, \beta)}}\log p+\sum_{\substack{p^{\nu}\leq \alpha+\beta-1 \\ g(p^{\nu}) \in \mathcal{B}(\alpha, \beta-\lfloor\alpha+\beta\rfloor)}}\log p=\frac{1}{\alpha}\sum_{p^{\nu}\leq N}\log p+\xi(N,q),
\end{equation}
where
\begin{equation}\label{thm1eq2}
\xi(N,q) \leq_{\varepsilon} N^{\varepsilon}(Nq^{-\gamma}+N^{1-\gamma/2}+q^{\gamma}N^{1-k\gamma}
+q^{\frac{\gamma}{1-\gamma}}N^{\frac{1-(k+1)\gamma}{1-\gamma}})
\end{equation}

\vspace{2mm}
\noindent
The number of prime powers $p^{\nu} \leq N$ with $\nu \geq 2$ is $O(\pi(N^{1/2})),$ 
thus we have
\begin{align}\label{thm1est}
\sum_{\substack{p \leq N \\ g(p) \in \mathcal{B}(\alpha, \beta)}}\log p = &\frac{1}{\alpha}\sum_{p \leq N}\log p +O\bigg(\frac{N^{1/2}}{\alpha \log N}\bigg)\\\nonumber
&+O\bigg(N^{\varepsilon}(Nq^{-\gamma}+N^{1-\gamma/2}+q^{\gamma}N^{1-k\gamma}+q^{\frac{\gamma}{1-\gamma}}N^{\frac{1-(k+1)\gamma}{1-\gamma}})\bigg).
\end{align}
Suppose we assume $\frac{a_{k}}{\alpha}$ is an irrational number of finite type $t.$ Using Dirichlet's approximation theorem with $Q= N^{\frac{kt}{t+1}}$, we obtain a rational $p/q$ with $1 \leq q \leq N^{\frac{kt}{t+1}}$ such that 
\begin{equation}\label{Dricheq}
\bigg\vert \frac{a_{k}}{\alpha}-\frac{p}{q} \bigg\vert \leq \frac{1}{q N^{\frac{kt}{t+1}}}.
\end{equation}
By definition of finite type of irrational, for any positive $\varepsilon,$ there is positive constant $c$ such that
\begin{equation}\label{typeeq}
\bigg\vert \frac{a_{k}}{\alpha}-\frac{p}{q} \bigg\vert \geq \frac{c}{q^{t+1+\varepsilon}}.
\end{equation}
Then by \eqref{Dricheq} and \eqref{typeeq} there exists a convergent to the simple continued fraction expansion of $\frac{a_{k}}{\alpha}$ whose  denominator satisfies
\begin{equation}\label{denomeq}
 N^{\frac{k}{t+1}+\varepsilon} \ll q \leq N^{\frac{kt}{t+1}}.
\end{equation}
Therefore by \eqref{thm1est} and \eqref{denomeq} we obtain
\begin{align*}
\sum_{\substack{p \leq N \\ g(p) \in \mathcal{B}(\alpha, \beta)}}\log p = \frac{1}{\alpha}\sum_{p \leq N}\log p +O(N^{1-\frac{k\gamma}{t+1}+\varepsilon}+ N^{1-\frac{\gamma}{2}+\varepsilon}).
\end{align*}
This completes the proof of the Theorem \ref{thm1}.
\vskip 5mm \noindent
{\it Proof of Theorem \ref{thm2}:}
By \eqref{thm1eq1} and \eqref{thm1eq2} we have
\begin{equation}\label{thm2eq1}
\sum_{\substack{p^{\nu} \leq N \\ g(p^{\nu}) \in \mathcal{B}(\alpha, \beta)}}\log p+\sum_{\substack{p^{\nu}\leq \alpha+\beta-1 \\ g(p^{\nu}) \in \mathcal{B}(\alpha, \beta-\lfloor\alpha+\beta\rfloor)}}\log p=\frac{1}{\alpha}\sum_{p^{\nu}\leq N}\log p+\xi(N,q),
\end{equation}
where
\begin{equation}\label{thm2eq2}
\xi(N,q) \leq_{\varepsilon} N^{\varepsilon}(Nq^{-\gamma}+N^{1-\gamma/2}+q^{\gamma}N^{1-k\gamma}
+q^{\frac{\gamma}{1-\gamma}}N^{\frac{1-(k+1)\gamma}{1-\gamma}})
\end{equation}

\vspace{2mm}
\noindent
By Lemma \ref{lem2}, the second sum on the left hand side of \eqref{thm2eq1} is $< 1.04 ~(\alpha+\beta-1).$ \\
Therefore, we have
$$\sum_{\substack{p \leq N \\ g(p) \in \mathcal{B}(\alpha, \beta)}}\log p \geq \frac{1}{\alpha}\sum_{p\leq N}\log p+\xi(N,q)-1.04(\alpha+\beta-1)+\bigg(\frac{1}{\alpha}-1\bigg)\sum_{\substack{p^{\nu}\leq N \\ \nu \geq 2}}\log p.$$
 Notice that the last term is negative, it is obviously bounded by
 $$\bigg(1-\frac{1}{\alpha}\bigg)\sum_{\substack{p^{\nu}\leq N \\ \nu \geq 2}}\log p < \bigg(1-\frac{1}{\alpha}\bigg)\pi(N^{1/2})\log N.$$
 We will use inequality $(2.18)$ Rosser and Schoenfeld \cite{RS62} for $\pi(x)$, we have
 $$\bigg(1-\frac{1}{\alpha}\bigg)\sum_{\substack{p^{\nu}\leq N \\ \nu \geq 2}}\log p < \bigg(1+\frac{3}{\log N}\bigg)N^{1/2}$$
we also use inequality (3.16) of Rosser and Schoenfeld which is,
$$\sum_{p \leq N}\log p > N - \frac{N}{\log N}$$ for $N \geq 41$.
Therefore we obtain
$$\sum_{\substack{p\leq N \\ g(p) \in \mathcal{B}(\alpha, \beta)}}\log p \geq \frac{N}{\alpha}\bigg(1-\frac{1}{\log N}\bigg)+\xi(N,q)-1.04(\alpha+\beta-1)-\bigg(1+\frac{3}{\log N}\bigg)N^{1/2}.$$
We thus find a prime $p \leq N$ and $p^{2} \in \mathcal{B}(\alpha,\beta)$ if we show that the following inequality
$$\frac{N}{\alpha}\bigg(1-\frac{1}{\log N}\bigg) > \xi(N,q)+1.04(\alpha+\beta-1)+\bigg(1+\frac{3}{\log N}\bigg)N^{1/2},$$
which we may also replace by
$$0.73 \frac{N}{\alpha} > 1.04(\alpha+\beta-1)+1.81N^{1/2}+\xi(N,q).$$
By \eqref{thm2eq2} we have,
$$ 0.73 > 1.04\frac{\alpha}{N}(\alpha+\beta-1)+1.81\frac{\alpha}{N^{1/2}}+C(\varepsilon)N^{\varepsilon}\alpha(q^{-\gamma}+N^{-\gamma/2}+q^{\gamma}N^{-k\gamma}
+q^{\frac{\gamma}{1-\gamma}}N^{\frac{-k\gamma}{1-\gamma}})$$
and appropriate absolute constant $C(\varepsilon)$ depending only on $\varepsilon$
but not an $\alpha$.

\vspace{2mm}
\noindent
Obviously $N$ need to be larger than Max$\lbrace \alpha^{2/\gamma}, B \rbrace$ and q larger than $\alpha^{1/\gamma}$. We shall take both N and q somewhat larger so that above inequality holds, now choose
$$N=\alpha^{\frac{2k-1}{k\gamma}}B^{\frac{k-1}{k}}q^{\frac{1}{k}}\eta^{\frac{\varepsilon}{\gamma}}, \quad
q=\alpha^{\frac{\gamma+1}{\gamma}} B \eta$$
with some large parameter $\eta$ to be specified later and B=max$\lbrace 1,\beta \rbrace$. Then the latter inequality can be rewritten as
\begin{align*}
&0.73 > 1.04(\alpha+\beta-1)(\alpha^{\frac{(k-1)\gamma-2k}{k\gamma}}B^{-1}\eta^{-\frac{1}{k}-\frac{\varepsilon}{\gamma}})+1.81(\alpha^{\frac{(2k-1)\gamma-2k}{2k\gamma}}B^{-1/2}\eta^{-\frac{1}{2k}-\frac{\varepsilon}{2\gamma}})\\
&+C \bigg(\alpha^{-\gamma+\frac{(2k+\gamma)\varepsilon}{k\gamma}}B^{-\gamma+\varepsilon}\eta^{-\gamma+\frac{\varepsilon}{k}+\frac{\varepsilon^{2}}{\gamma}}+\alpha^{-\frac{\gamma}{2k}+\frac{(2k+\gamma)\varepsilon}{k\gamma}}B^{-\frac{\gamma}{2}+\varepsilon}\eta^{-\frac{\gamma}{2k}+\frac{(2-k)\varepsilon}{2k}+\frac{\varepsilon^{2}}{\gamma}}\\
&+\alpha^{2-2k+\frac{(2k+\gamma)\varepsilon}{k\gamma}}B^{(1-k)\gamma+\varepsilon}\eta^{\frac{(1-k^{2})\varepsilon}{k}+\frac{\varepsilon^{2}}{\gamma}}+\alpha^{\frac{2(1-k)-\gamma}{1-\gamma}+\frac{(2k+\gamma)\varepsilon}{k\gamma}}B^{\frac{(1-k)\gamma}{1-\gamma}+\varepsilon}\eta^{\frac{(1-k^{2}-\gamma)\varepsilon}{1-\gamma}+\frac{\varepsilon^{2}}{\gamma}}\bigg).
\end{align*}

\vspace{2mm}
\noindent
Since $k \geq 2$ and $\gamma=4^{1-k}$ assuming $\varepsilon < \frac{\gamma^{2}}{2(2k+\gamma)}$, as we may, all exponents of $\alpha,B$ and $\eta$ are negative. Therefore the above inequality is satisfied for all sufficiently large $\eta$, say $\eta \geq \eta_{0}.$ Since $\eta$ is interwined with $q$ a little care needs to be taken. In order to find a suitable $\eta$ recall
$\alpha$ is irrational. Hence, by Dirichlet's approximation theorem, there are infinitely many solution $\frac{a}{q}$ to inequality $|\alpha- \frac{a}{q}| < \frac{1}{q^{2}};$ in view of $\mathbf{a}=\frac{1}{\alpha}$ we may take the reciprocals of the convergents $\frac{p_{n}}{q_{n}}$ to the continued fraction expansion of $\alpha.$ We shall choose $l$ such that $\eta_{0} \leq \frac{p_{m+l}}{p_{m}},$ where $m$ is defined by \eqref{cngtineq}, for then the choice $q=p_{m+l}$ will yield an $\eta \geq \eta_{0}.$ The choice of $\eta$ follows from $(12)$ of \cite{ST16}. Therefore the choice of $l$ as it depends on $\eta.$ This completes the proof of the Theorem \ref{thm2}. 
\section{Proof of theorem 3}\label{section5}
\noindent
The present section is devoted to a proof of Theorem \ref{thm3}.
\vskip 3mm \noindent
{\it Proof of Theorem \ref{thm3}:} 
It follows from Proposition \ref{teclemma2} and partial summation formula
\begin{equation*}
\sum_{1 \leq |l| \leq L}C_{l} \sum_{\substack{n\leq N \\ n \equiv f(d)}} \Lambda(n)e(ln/\alpha) \ll_{\varepsilon} (NL)^{\varepsilon} \bigg(\frac{N}{q^{1/2}}+N^{1/2}q^{1/2}+\frac{N^{1/2}q^{1/2}}{L^{1/2}}+N^{3/4}d^{1/2}+\frac{N^{4/5}}{d^{1/5}}\bigg).
\end{equation*} 
By \eqref{fourier2} and \eqref{prop2eq}, we obtain 
\begin{align*}
&\sum_{\substack{n\leq N \\ n \equiv f(d)}} \Lambda(n)\chi\bigg(\frac{n}{\alpha}+\frac{1-2\beta}{2\alpha}\bigg)=\bigg(\frac{1}{\alpha}\bigg)\sum_{\substack{n\leq N \\ n \equiv f(d)}} \Lambda(n)+\frac{1}{L} \frac{N}{\varphi(d)}\\
&+O\bigg((NL)^{\varepsilon} \bigg(\frac{N}{q^{1/2}}+N^{1/2}q^{1/2}+\frac{N^{1/2}q^{1/2}}{L^{1/2}}+N^{3/4}d^{1/2}+\frac{N^{4/5}}{d^{1/5}}\bigg)\bigg).
\end{align*}
Choose
\begin{equation*}
L= \frac{N}{qd^{2}}.
\end{equation*}

\vspace{2mm}
\noindent
Therefore we obtain an estimate
\begin{align*}
&\sum_{\substack{n\leq N \\ n \equiv f(d)}} \Lambda(n)\chi\bigg(\frac{n}{\alpha}+\frac{1-2\beta}{2\alpha}\bigg)=\bigg(\frac{1}{\alpha}\bigg)\sum_{\substack{n\leq N \\ n \equiv f(d)}} \Lambda(n)\\
&+O\bigg(N^{\varepsilon} \bigg(\frac{N}{q^{1/2}}+N^{1/2}q^{1/2}+qd+N^{3/4}d^{1/2}+\frac{N^{4/5}}{d^{1/5}}\bigg)\bigg).
\end{align*}
We rewrite above equality as
\begin{align*}
\sum_{\substack{p^{\nu}\leq N \\ p^{\nu} \in \mathcal{B}(\alpha, \beta) \\ p^{\nu} \equiv f(d)}}\log p &+\sum_{\substack{p^{\nu}\leq \alpha+\beta-1 \\ p^{\nu} \in \mathcal{B}(\alpha, \beta-\lfloor\alpha+\beta\rfloor) \\ p^{\nu} \equiv f (d)}}\log p=\frac{1}{\alpha}\sum_{\substack{p^{\nu}\leq N \\ p^{\nu} \equiv f(d)}} \log p\\
&+O\bigg(N^{\varepsilon} \bigg(\frac{N}{q^{1/2}}+N^{1/2}q^{1/2}+qd+N^{3/4}d^{1/2}+\frac{N^{4/5}}{d^{1/5}}\bigg)\bigg).
\end{align*}
By $(4.3.3)$ of \cite{RR} we have
\begin{equation*}
\frac{1}{\alpha}\sum_{\substack{p^{\nu}\leq N \\ p \equiv f (d)\\ \nu \geq 2}}\log p \leq  \frac{1.0012}{\alpha} (\sqrt{N}+N^{1/3}).
\end{equation*}
Thus we have
\begin{align}\label{thm3est}
\sum_{\substack{p \leq N \\ p \in \mathcal{B}(\alpha, \beta)\\ p \equiv f (d)}}\log p =&\frac{1}{\alpha}\sum_{\substack{p \leq N \\ p \equiv f(d)}} \log p+O\bigg(\frac{N^{1/2}}{\alpha}\bigg)\nonumber\\ 
&+O\bigg(N^{\varepsilon} \bigg(\frac{N}{q^{1/2}}+N^{1/2}q^{1/2}+qd+N^{3/4}d^{1/2}+\frac{N^{4/5}}{d^{1/5}}\bigg)\bigg).
\end{align}
Suppose we assume $\frac{1}{\alpha}$ is an irrational number of finite type $t$. By using Dirichlet's approximation theorem with $Q= N^{\frac{t}{1+t}}$, we obtain a rational $p/q$ with $1 \leq q \leq N^{\frac{t}{1+t}}$ such that 
\begin{equation}\label{lemeq}
\bigg\vert \frac{1}{\alpha}-\frac{p}{q} \bigg\vert \leq \frac{1}{q N^{\frac{t}{t+1}}}.
\end{equation}
And by definition of finite type of irrational, for any positive $\varepsilon,$ there is positive constant $c$ such that
\begin{equation}\label{defeq}
\bigg\vert \frac{1}{\alpha}-\frac{p}{q} \bigg\vert \geq \frac{c}{q^{t+1+\varepsilon}}.
\end{equation}
Then by \eqref{lemeq} and \eqref{defeq} there exists a convergent to the simple continued fraction expansion of $\frac{1}{\alpha}$ whose  denominator satisfies
\begin{equation}\label{combeq}
 N^{\frac{1}{1+t}+\varepsilon} \ll q \leq N^{\frac{t}{1+t}}.
\end{equation}
Therefore by \eqref{thm3est} and \eqref{combeq} we obtain
\begin{align*}
\sum_{\substack{p \leq N \\ p \in \mathcal{B}(\alpha, \beta)\\ p \equiv f (d)}}\log p =\frac{1}{\alpha}\sum_{\substack{p \leq N \\ p \equiv f(d)}} \log p
+O(N^{\varepsilon}(N^{1-\frac{1}{2(1+t)}}+N^{3/4} d^{1/2}+N^{4/5}d^{-1/5})).
\end{align*}
This completes the proof of the Theorem \ref{thm3}.

\vspace{2mm}
\noindent
\section{Proof of proposition 2}\label{section6}
Proof of Proposition \ref{teclemma2} is based on work of Balog and Perelli \cite{BP85}.
\subsection{Some Lemmas}
Here we list several lemmas required for the proof. The following lemma gives explicit bound for average of von Mangoldt function 

\begin{lemma}\cite{RS62} \label{lem2}
For any $N \in \mathbb{N}$ $$ \sum_{n \leq N} \Lambda(n) \leq c_{0} N,$$ for some constant $c_{0}$, where one may take $c_{0}=1.04.$
\end{lemma}

\begin{lemma}\label{lem3}
$$ \sum_{\substack{x<m\leq x'\\ m \equiv f(d)}}e(m\theta) \ll \min \bigg(\frac{x'}{d}+1,\Vert\theta d\Vert^{-1}\bigg).$$
\end{lemma}

\begin{lemma}\cite{Vino}\label{lem4}
Suppose that X, Y $\geq 1$ are positive integers, Also suppose that $|\alpha-a/q|< q^{-2}$, where $\alpha$ is a real number, $a$ and $q$ integers satisfying $(a,q)=1$. Then
\begin{equation*}
\sum_{x \leq X}\min(Y,\Vert\alpha x\Vert^{-1}) \ll \frac{XY}{q}+(X+q) \log 2q,
\end{equation*}
\begin{equation*}
\sum_{x \leq X}\min\bigg(\frac{XY}{x},\Vert\alpha x\Vert^{-1}\bigg) \ll \frac{XY}{q}+(X+q) \log (2XYq).
\end{equation*}
\end{lemma}

\begin{lemma}\cite{VAN}\label{lem5}
For any real number $\vartheta$ and natural numbers $N, l$ and $1 \leq f < d$ such that $(f,d)=1,$ we have 
\begin{equation*}
\sum_{\substack{n \leq N \\ n \equiv f (d)}} \Lambda (n) e(ln\vartheta) = O(N^{1/2})+S_{1}-S_{2}-S_{3},
\end{equation*}
where 
\begin{align*}
S_{1} &= \mathop{\sum_{m \leq U}\sum_{n \leq N/m}}_{mn \equiv f (d)} \mu(m) (\log n)  e(lmn \vartheta),\\
S_{2} &= \mathop{\sum_{m \leq U^{2}}\sum_{n \leq N/m}}_{mn \equiv f (d)} \phi_{1}(m)  e(lmn \vartheta),\\
S_{3} &= \mathop{\sum_{U < m \leq N/U}\sum_{U < n \leq N/m}}_{mn \equiv f (d)} \phi_{2}(m) \Lambda(n) e(lmn \vartheta),
\end{align*}
and 
\begin{equation*}
\phi_{1}(m) \ll \log m, \quad \phi_{2}(m) \ll d_{2}(m).
\end{equation*}
Here $U$ is an arbitrary parameters to be chosen later
satisfying $1 \leq U \leq N^{1/2}.$ 
\end{lemma}

\begin{lemma}\label{lem6}
Suppose that $\varepsilon > 0$ and that $\phi(u)$ and $\psi(v)$ are real valued functions such that $|\phi(u)| \ll T,$ $|\psi(v)| \ll F.$ Suppose that $|\vartheta-a/q|< q^{-2}$, where $\vartheta$ is a real number, $a$ and $q$ integers satisfying $(a,q)=1$. For positive integers $N, W, X, $ and $L$ write 
\begin{equation}\label{S sum}
S = \sum_{|l| \leq L} \bigg|\mathop{\sum_{X < v \leq 2X}\sum_{u \leq W}}_{\substack{uv \leq N \\ uv \equiv f (d)}} \phi(u) \psi(v) e(luv \vartheta) \bigg|.
\end{equation}
Then
\begin{equation*}
S \ll TF \bigg(\frac{LWX^{1/2}}{d^{1/2}}+(LXd)^{\varepsilon} \bigg(\frac{LXW}{q^{1/2}}+LXW^{1/2}d^{1/2}+L^{1/2}q^{1/2}X^{1/2}W^{1/2}\bigg)\bigg).
\end{equation*}
\end{lemma}
\begin{proof}
For the moment we shall ignore the condition $uv \leq N$ in \eqref{S sum}. Consider
\begin{equation}
S = \sum_{|l| \leq L} \bigg|\mathop{\sum_{X < v \leq 2X}\sum_{u \leq W}}_{uv \equiv f (d)}\phi(u) \psi(v) e(luv \vartheta) \bigg|.
\end{equation}
We observe that 
\begin{equation}\label{S f1f2 eq}
S= \sum_{\substack{f_{1}f_{2} \equiv f (d) \\ (f_{1},d)=(f_{2},d)=1}}R_{f_{1},f_{2}} \ll d \max_{\substack{f_{1}f_{2} \equiv f (d) \\ (f_{1},d)=(f_{2},d)=1}} |R_{f_{1},f_{2}}|,
\end{equation}
where
\begin{equation*}
R_{f_{1},f_{2}} = \sum_{|l| \leq L} \bigg| \sum_{\substack{X < v \leq 2X \\ v \equiv f_{2}(d)}}\sum_{\substack{u \leq W \\ u \equiv f_{1}(d)}} \phi(u) \psi(v) e(luv \vartheta)\bigg|.
\end{equation*}
By using Cauchy Schwarz inequality we obtain
\begin{align}\label{R1}
|R_{f_{1},f_{2}}|^{2} &\leq L \sum_{|l| \leq L} \bigg| \sum_{\substack{u \leq W \\ u \equiv f_{1}(d)}} \sum_{\substack{X < v \leq 2X \\ v \equiv f_{2}(d)}}\phi(u) \psi(v) e(luv \vartheta)\bigg|^{2}\nonumber\\
& \leq \sum_{|l| \leq L} \bigg(\sum_{\substack{u \leq W \\ u \equiv f_{1}(d)}}|\phi(u)|^{2}\bigg) \bigg(\sum_{\substack{u \leq W \\ u \equiv f_{1}(d)}} \bigg| \sum_{\substack{X < v \leq 2X \\ v \equiv f_{2}(d)}} \psi(v) e(luv \vartheta)\bigg|^{2}\bigg)\nonumber\\
& \leq \frac{T^{2}LW}{d} \bigg(\frac{F^{2}LWX}{d^{2}}+R_{1}\bigg),
\end{align}
where 
\begin{equation*}
R_{1} = \sum_{|l| \leq L} \sum_{\substack{u \leq W \\ u \equiv f_{1}(d)}}\sum_{\substack{X < v_{1},v_{2} \leq 2X \\ v_{1},v_{2} \equiv f_{2}(d) \\ v_{1} \neq v_{2}}} \psi(v_{1})\psi(v_{2}) e(lu(v_{1}-v_{2})\vartheta).
\end{equation*}
We may write $R_{1}$ in the form
\begin{align} \label{R1 eq1}
&= \sum_{|l| \leq L} \sum_{\substack{u \leq W \\ u \equiv f_{1}(d)}}\sum_{\substack{|k| \leq 2X \\ k \equiv 0(d)}}\sum_{\substack{X < v_{1},v_{2} \leq 2X \\ v_{1},v_{2} \equiv f_{2}(d) \\ v_{1} \neq v_{2} \\ v_{1}-v_{2}=k}} \psi(v_{1})\psi(v_{2}) e(luk\vartheta)\nonumber\\
&\leq \sum_{|l| \leq L} \sum_{\substack{|k| \leq 2X \\ k \equiv 0(d)}}\zeta_{1}(k) \sum_{\substack{u \leq W \\ u \equiv f_{1}(d)}} e(luk\vartheta),
\end{align}
where
\begin{equation*}
\zeta_{1}(k)= \sum_{\substack{X < v_{1},v_{2} \leq 2X \\ v_{1},v_{2} \equiv f_{2}(d) \\ v_{1} \neq v_{2} \\ v_{1}-v_{2}=k}} \psi(v_{1})\psi(v_{2}) \ll \frac{F^{2}X}{d}.
\end{equation*}
We apply Lemma \ref{lem3} for innermost sum of \eqref{R1 eq1}, we get
\begin{equation*}
|R_{1}| \leq \frac{F^{2}X}{d} \sum_{|l| \leq L} \sum_{\substack{|k| \leq 2X \\ k \equiv 0(d)}} \min \bigg (\frac{W}{d}+1, \Vert lkd\vartheta\Vert)^{-1}\bigg). 
\end{equation*} 
Let $r= lkd$ so that $1 \leq |r| \leq 2LXd$ and $r$ will run through all the integers in the interval above, also number of representations of $r$ is not more than $d_{2}(r).$ Therefore we have
\begin{equation*}
|R_{1}| \leq \frac{F^{2}X}{d} (LXd)^{\varepsilon} \sum_{|r| \leq 2LXd} \min \bigg (\frac{W}{d}+1, \Vert r\vartheta\Vert)^{-1}\bigg).
\end{equation*}
Then by using Lemma \ref{lem4} we obtain
\begin{equation}\label{R1 est}
R_{1} \ll F^{2}(LXd)^{\varepsilon} \bigg(\frac{LX^{2}W}{qd}+LX^{2}+\frac{qX}{d}\bigg).
\end{equation}
By \eqref{R1} and \eqref{R1 est} we have
\begin{equation}\label{S f1f2 est}
R_{f_{1},f_{2}} \ll TF \bigg(\frac{LWX^{1/2}}{d^{3/2}}+(LXd)^{\varepsilon} \bigg(\frac{LXW}{q^{1/2}d}+\frac{LXW^{1/2}}{d^{1/2}}+\frac{L^{1/2}q^{1/2}X^{1/2}W^{1/2}}{d}\bigg)\bigg).
\end{equation}
Thus Lemma follows from \eqref{S f1f2 eq} and \eqref{S f1f2 est} .
\end{proof}
\begin{lemma}\label{lem7}
Suppose we have the hypotheses and notations of Lemma \ref{lem6} with either $\phi(x) =1$ or $\phi(x) =\log x$   for all x. Then 
\begin{equation}\label{lem7 est}
S \ll F (LXd)^{\varepsilon} (LXWq^{-1}+ LXd+q).
\end{equation}
\begin{proof}
The $\log x$ factor may easily be removed by partial summation formula so we presume that $\phi(x) \equiv 1.$ Again we may ignore the condition $uv \leq N.$ Therefore we need to estimate
\begin{align*}
S &= \sum_{|l| \leq L} \bigg|\mathop{\sum_{X < v \leq 2X}\sum_{u \leq W}}_{uv \equiv f (d)}\psi(v) e(luv \vartheta) \bigg|\\
&\leq F \sum_{|l| \leq L} \sum_{\substack{X \leq v < 2X \\ (v,d)=1}} \bigg|\sum_{\substack{v \leq W \\ v \equiv f\bar u (d)}}e(luv \vartheta)\bigg|,
\end{align*}
where $\bar u$ is defined by $u \bar u \equiv 1 \pmod d.$ Then by using Lemma \ref{lem3}, we have 
\begin{equation*}
S \leq F \sum_{|l| \leq L} \sum_{\substack{X \leq v < 2X \\ (v,d)=1}} \min \bigg(\frac{W}{d}+1, \Vert lvd \vartheta \Vert^{-1} \bigg).
\end{equation*}
Let $r= lvd$ so that $1 \leq |r| \leq 2LXd$ and $r$ will run through all the integers in the interval above, also number of representations of $r$ is not more than $d_{2}(r).$ Therefore we have
\begin{equation} \label{lem7 final}
S \leq F (LXd)^{\varepsilon} \sum_{|r| \leq 2LXd} \min \bigg(\frac{W}{d}+1, \Vert r \vartheta \Vert^{-1} \bigg).
\end{equation}
Thus \eqref{lem7 est} follows from \eqref{lem7 final} and Lemma \ref{lem4}.
\end{proof}
\end{lemma}

\vspace{2mm}
\noindent
\textit{Proof of the proposition \ref{teclemma2}:}
We may assume that
\begin{equation}
 N \geq \max(qd^{2}L^{-1}, d^{6}), \quad q \geq d^{2} 
\end{equation}  
otherwise \eqref{prop2eq} is a consequence of the trivial bound,
\begin{equation*}
S(\vartheta) \leq \frac{LN}{d}.
\end{equation*}
Using Lemma \ref{lem5} we have the following sums to estimate 
\begin{align*}
S'_{1} &= \sum_{|l| \leq L}\bigg|\mathop{\sum_{m \leq U}\sum_{n \leq N/m}}_{mn \equiv f (d)} \mu(m) (\log n)  e(lmn \vartheta)\bigg|\\
S'_{2} &= \sum_{|l| \leq L}\bigg|\mathop{\sum_{m \leq U^{2}}\sum_{n \leq N/m}}_{mn \equiv f (d)} \phi_{1}(m)  e(lmn \vartheta)\bigg|\\
S'_{3} &= \sum_{|l| \leq L}\bigg|\mathop{\sum_{U < m \leq N/U}\sum_{U < n \leq N/m}}_{mn \equiv f (d)} \phi_{2}(m) \Lambda(n) e(lmn \vartheta)\bigg|.
\end{align*}
By dyadic division we write 
\begin{equation*}
S'_{1}= \sum_{t=0}^{\big[\frac{\log U}{\log 2}\big]} S_{1t},
\end{equation*}
where 
\begin{equation*}
S_{1t}= \sum_{l \leq L} \bigg|\mathop{\sum_{2^{t} < m \leq 2^{t+1}}\sum_{n \leq N/m}}_{mn \equiv f (d)} \mu(m) (\log n)  e(lmn \vartheta)\bigg|.
\end{equation*}
Then using Lemma \ref{lem7}, we  get 
\begin{equation*}
S'_{1} \ll  (NL)^{\varepsilon} (LNq^{-1}+LUd+q).
\end{equation*}

\noindent
$S'_{2}$ can be estimated similarly as $S'_{1}$ by partitioning into dyadic subsums say $S_{2t}$. We estimate $S_{2t}$ using Lemma \ref{lem7}, and we get
\begin{equation*}
S'_{2} \ll (NL)^{2\varepsilon} (LNq^{-1}+LU^{2}d+q).
\end{equation*}
We write $S'_{3}=S'_{31}+S'_{32},$ where
\begin{align*}
S'_{31} &=\sum_{l \leq L}\bigg|\mathop{\sum_{U < m \leq N^{1/2}}\sum_{U < n \leq N/m}}_{mn \equiv f (d)} \phi_{2}(m) \Lambda(n) e(lmn \vartheta)\bigg|,\\
S'_{32} &=\sum_{l \leq L}\bigg|\mathop{\sum_{U < n \leq N^{1/2}}\sum_{N^{1/2} < m \leq N/n}}_{mn \equiv f (d)} \phi_{2}(m) \Lambda(n) e(lmn \vartheta)\bigg|.
\end{align*}

\noindent
By dividing $S'_{31}$ dyadically we obtain
\begin{equation*}
S'_{31}= \sum_{t=0}^{R} S_{31t},
\end{equation*}
where 
\begin{equation*}
S_{31t}= \sum_{l \leq L} \bigg|\mathop{\sum_{U2^{t} < m \leq U2^{t+1}}\sum_{U< n \leq N/m}}_{mn \equiv f (d)} \mu(m) (\log n)  e(lmn \vartheta)\bigg| \ \text{ and } \ R=\bigg[\frac{\log (\frac{N^{1/2}}{U})}{\log 2}\bigg].
\end{equation*}
Then using Lemma \ref{lem6}
\begin{equation*}
S'_{31} \ll (NL)^{3\varepsilon} \bigg(\frac{LN}{q^{1/2}}+LN^{1/2}q^{1/2}+LN^{3/4}d^{1/2}+\frac{LN}{U^{1/2}d^{1/2}}\bigg).
\end{equation*}
Similarly we can show that $S'_{32}$ has the same upper bound. Therefore
\begin{equation}\label{prop2fineq}
S(\vartheta) \ll_{\varepsilon} (NL)^{\varepsilon} \bigg(\frac{NL}{q^{1/2}}+L^{1/2}N^{1/2}q^{1/2}+LN^{3/4}d^{1/2}+LU^{2}d+q+\frac{LN}{U^{1/2}d^{1/2}}\bigg).
\end{equation}
Then \eqref{prop2eq} follows from \eqref{prop2fineq} with the chioce of 
\begin{equation*}
U= \frac{N^{2/5}}{d^{3/5}}
\end{equation*} 
and the observation $q \leq L^{1/2}N^{1/2}q^{1/2}.$ 

\bibliographystyle{plain}  
\bibliography{refs_Beatty}
\end{document}